\documentclass[10pt]{article}

\usepackage{amsrefs}
\usepackage{amsmath,amssymb,amsthm}
\usepackage{enumerate,pxfonts}
\usepackage{tikz,graphicx,pxfonts}
\usepackage{color,framed,amsbsy}
\usepackage[all,cmtip]{xy}
\usepackage[applemac]{inputenc}
\usepackage{tocloft,MnSymbol}
\usepackage[colorlinks=true, allcolors=black]{hyperref}

\usepackage[OT1]{fontenc}
\DeclareFontFamily{OT1}{pzc}{}
\DeclareFontShape{OT1}{pzc}{m}{it}{<-> s * [1.35] pzcmi7t}{}
\DeclareMathAlphabet{\mathcal}{OT1}{pzc}{m}{it}

\newcommand \aft{\circ}

\newcommand \bs{\backslash}
\newcommand \C{{\mathbb C}}
\newcommand \CA{\mathcal{A}}
\newcommand \CB{\mathcal{B}}

\newcommand \CF{\mathcal{F}}

\newcommand \CI{\mathcal{I}}
\newcommand \CJ{\mathcal{J}}

\newcommand \CT{\mathcal{T}}

\newcommand \Ga{\Gamma}

\newcommand \ga{\gamma}

\newcommand \HS{\mathcal{HS}}
\newcommand \Id{{\rm Id}}

\newcommand \Isom{\operatorname{Isom}}

\newcommand \mqed{\tag*\qedhere}

\newcommand \ol{\overline}

\newcommand \Rinj{\operatorname{InjR}}

\newcommand \sm{\smallsetminus}
\newcommand \supp{\operatorname{supp}}
\newcommand \tr{\operatorname{tr}}

\newcommand \vol{\operatorname{vol}}

\renewcommand \({\big(}
\renewcommand \){\big)}
\renewcommand \[{\left(}
\renewcommand \]{\right)}

\newcommand{\e}
[1]{\emph{#1}\index{#1}}

\newcommand{\tto}
[1]{\stackrel{#1}{\longrightarrow}}

\newtheorem{theorem}{Theorem}[section]

\newtheorem{lemma}[theorem]{Lemma}
\newtheorem{corollary}[theorem]{Corollary}
\newtheorem{proposition}[theorem]{Proposition}

\theoremstyle{definition}
\newtheorem{definition}[theorem]{Definition}
\newtheorem{example}[theorem]{Example}

\newtheorem{remark}[theorem]{Remark}

\begin{document}

\pagestyle{myheadings} \markright{Benjamini-Schramm and spectral convergence II}

\title{
\mbox{Benjamini-Schramm and spectral convergence II.}\\ The non-homogeneous case}
\author{Anton Deitmar}
\date{}
\maketitle

{\bf Abstract:}
The equivalence of spectral convergence and Benjamini-Schramm convergence is extended from homogeneous spaces to spaces which are compact modulo isometry group.
The equivalence is proven under the condition of a uniform discreteness property.
It is open, which implications hold without this condition.  

$$ $$

\section*{Introduction}

The notion of Benjamini-Schramm convergence (BS-convergence) originally was defined for  graphs see \cite{BS-original} or \cites{Abert1,Abert2,Abert3,Anan,Csik}.
It then got extended to complexes, surfaces and other manifolds \cites{Abert4,BS-AD,Gelander,Kionke,Levit,Mohammadi,Monk}.
Relations to number theory were revealed
\cites{Brumley,Fracyk}.
In \cite{LMS} it was brought into relation with quantum ergodicity.
Roughly, a sequence $(X_n)$ of  metric spaces is said to converges to a space $X$, if for any $R>0$ the probability of a ball in $X_n$ of radius $R$ being isomorphic with a ball in $X$ converges to 1 as $n\to\infty$.

In the paper \cite{7Samurais}, BS-convergence is 
used in the context of locally symmetric Riemannian manifolds of the form  $\Ga\bs G/K$, for a semisimple Lie group $G$, a maximal compact subgroup $K$ and a discrete subgroup $\Ga$.
In this context, a Benjamini-Schramm convergent sequence is the same as a \e{Farber sequence}, see \cites{Carderi,Farber}.
In \cite{7Samurais} it is shown, among other things, that the normalised spectral measures of a uniformly discrete sequence $(\Ga_n)$ of lattices in a connected semi-simple Lie group $G$ without center, weakly converges to the Plancherel measure, if the sequence of Riemannian manifolds $\Ga_n\bs G/K$ is BS-convergent to the symmetric space $G/K$, where $K$ is a maximal compact subgroup of $G$.

In the  paper \cite{BS-AD}, this assertion, together with its  converse, is extended to subgroups of arbitrary locally compact groups.
In the present paper, we extend the latter to  proper metric spaces $(X,d)$.
In the case of a homogeneous space, this boils down to the notion of the previous paper for metrisable locally compact groups.
In general, equivalences are shown up to one open question. The results of the paper can be summarised in the diagram
$$
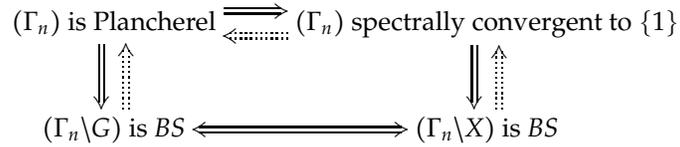
\xymatrix{
(\Ga_n)\text{ is Plancherel}\ar@<-1ex>@{=>}[d]\ar@<1ex>@{=>}[r]\ar@<-1ex>@{<::}[r]
	&(\Ga_n)\text{ spectrally convergent to }\{1\}\ar@<-1ex>@{=>}[d]\\
(\Ga_n\bs G)\text{ is }BS\ar@<-1ex>@{::>}[u]\ar@{<=>}[r]
	&(\Ga_n\bs X)\text{ is }BS\ar@<-1ex>@{::>}[u]\\
}
$$
The right hand side of the diagram depicts the results of \cite{BS-AD}, the other 6 implications are subject of this paper. Let's go through the notions of the diagram.
The dotted arrows are only shown to hold under the additional assumption of uniform convergence, see Definition \ref{def3.3}.
The vertical arrows actually need this condition, whereas it is not clear, whether the horizontal dotted arrow holds without uniform discreteness.
A sequence of lattices $(\Ga_n)$ is Plancherel, if the spectrum of every invariant operator on $\Ga_n\bs G$ converges to the spectrum of the same operator on  $G$. The sequence is called spectrally convergent to $\{1\}$, if the same holds for $X$ instead of $G$.
The letters BS stand for {Benjamini-Schramm} convergence.

\tableofcontents

\section{The isometry group}

We collect some basic facts. We include the (simple) proofs for the convenience of the reader.
Recall that a metric space $X$ is called \e{proper}, if every closed ball of finite radius is compact.
Then $X$ is complete, locally compact and separable.

\begin{definition}
Let $(X,d)$ be a proper metric space.
By an \e{invariant measure} we mean a non-zero Radon measure $\mu$, which is invariant under all isometries of $X$.
\end{definition} 

\begin{remark}
In the following, we fix an invariant  measure $\mu$ on $X$. We shall be interested in $L^2$-spaces, so we can replace $X$ with the support of $\mu$ and thus assume that $\supp(\mu)=X$.
In particular, for a function $0\le f\in C(X)$ the equation $\int_Xf(x)\,d\mu(x)=0$ implies $f=0$.
\end{remark}

\begin{proposition}\label{prop111}
Let $X$ be a proper metric space and let $G=\Isom(X)$ be its isometry group.
Equip $G$ with the compact-open topology. 
\begin{enumerate}[\rm (a)]
\item $G$ is a metrisable  locally compact  group.
\item Each orbit $Gx$ is closed.
\item The quotient map $q:X\to G\bs X$ is open.
\item The quotient space $G\bs X$ is locally compact and metrisable.
\item 
The map $C_c(X)\to C_c(G\bs X)$, $f\mapsto f^G$ with $f^G(x)=\int_Gf(gx)\, dx$, according to some Haar-measure on $G$, is surjective.
\item If $G$ is unimodular, then for every $G$-invariant Radon measure $\mu$ on $X$ there exists a uniquely determined Radon measure on $G\bs X$ such that for every $f\in C_c(X)$ one has
$$
\int_X f(x)\, dx=\int_{G\bs X}\int_Gf(gx)\, dg\, dx.
$$
\item For every $x\in X$ the map $G\to X$, $g\mapsto gx$ is proper.
\end{enumerate}
These assertions also hold true, if $G$ is replaced by a closed subgroup $H\subset G$.
\end{proposition}

\begin{proof}
(a) We fix some $x_0\in X$ and we claim that the topology is given by the metric
$$
d_{G,x_0}(f,g)=\sum_{n=1}^\infty\frac1{2^n}\sup_{x\in \ol B_n(x_0)} d(f(x),g(x)).
$$
The sum converges, since for $d(x,x_0)\le n$ we have
\begin{align*}
d(f(x),g(x))&\le d(f(x),f(x_0))+d(f(x_0),g(x_0))+d(g(x_0),g(x))\\
&=2d(x,x_0)+d(f(x_0)+g(x_0)).
\end{align*}
The sum defines a metric on $G$ and convergence in this metric means compact uniform convergence, hence the topology is the compact-open topology.
It remains to show that $G$ is locally compact.
As $G$ is a topological group, it suffices to show that the closed ball $\ol B_{G,1}(\Id)$ of radius one around the identity in $G$ is compact.
So let $f_j$ be a sequence in $\ol B_{G,1}(\Id)$.
Then in particular $d(f_j(x),x)\le 2^n$ for every $x\in\ol B_n(x_0)$ and so $d(f_j(x),x_0)\le 2^n+n$, i.e.,
$$
f_j\[\ol B_n(x_0)\]\subset \ol B_{n+2^n}(x_0).
$$
We shall construct a convergent subsequence.
For $n=0$ we have $\ol B_n(x_0)=\{ x_0\}$. The sequence $f_j(x_0)$ lies in the compact set $\ol B_1(x_0)$, hence has a convergent subsequence $f^{(0)}_j(x_0)$.
Next for $n\ge 1$ we assume given a subsequence $(f_j^{(n-1)})_j$ which converges pointwise on $\ol B_{n-1}(x_0)$.
We consider $f_j^{(n-1)}$ as an element of 
 the compact space $S=\prod_{x\in \ol B_n(x_0)}\ol B_{n+2^n}(x_0)$.
Let $A_j$ be the closure of $\{f_j^{(n-1)},f_{j+1}^{(n-1)},\dots\}$ in $S$. As $S$ is compact $A=\bigcap_jA_j$ is non-empty.
Let $a\in A_j$. Then there exists a subsequence $f_j^{(n)}$ of $f_j^{(n-1)}$ which converges to $a$ in $S$, which means that $f_j^{(n)}$ converges pointwise on $\ol B_n(x_0)$.
Then the subsequence $g_j=f_j^{(j)}$ converges pointwise to some $f:X\to X$ which is again an isometry.
As a pointwise convergent sequence of isometries between compact metric spaces converges uniformly, the claim follows.

(b)
Let $g_j$ be a sequence in $G$ such that $g_jx$ converges to some $z\in X$ and let $\delta=\max_jd(g_jx,z)$.
Let $A\subset X$ be a countable dense subset.
For $a\in A$, the sequence $g_ja$ lies in the compact set $\ol B_{d(x,a)+\delta}(z)$, hence has a convergent subsequence. A diagonal argument yields a subsequence, such that $g_ja$ converges to some $z_a$ for every $a\in A$. Then for every $y\in X$ the sequence $g_jy$ is Cauchy, hence convergent to some $z_y$. The map $h(y)=z_y$ is an isometry and is the pointwise limit of the $g_j$.
Hence $h(x)=z$ and so the orbit is closed.

(c)
Write $p:X\to G\bs X$ for the projection and let $U\subset X$ be open.
Then 
$$
p^{-1}\[p(U)\]=\bigcup_{g\in G} gU
$$
is open in $G$, which means that $p(U)$ is open in $G\bs X$.

(d)
Let $x\in X$ and let $U$ be a compact neighbourhood of $x$ in $X$.
Then $p(U)$ is a compact neighbourhood of $p(x)$, hence $G\bs X$ is locally compact.

We claim that
$$
d(Gx,Gy)=\inf_{g\in G}d(gx,y)
$$
defines a metric on $G\bs X$, which yields the quotient topology.
The properness of the metric $d$ together with part (b) implies that the infimum is attained, i.e., is a minimum.
This implies that $d_G$ is positive definite.
For the triangle inequality note that 
\begin{align*}
d_G(Gx,Gy)&=\inf_gd(gx,y)=\inf_h\inf_gd(gx,y)\\
&\le\inf_h\inf_g\(d(gx,hz)+d(hz,y)\)=d_G(x,z)+d_G(z,y).
\end{align*}
As the infimum is attained, it follows that $p(B_r(x))=B_r(Gx)$ for all $x\in X$, $r>0$ and this implies that the metric induces the quotient topology.

(e) and (f)
The proof proceeds as in Section 1.5 in \cite{HA2}.

(g) 
Let $x\in X$ and let $\phi:G\to X$, $g\mapsto gx$.
Let $C\subset X$ be compact. The orbit $Gx$ is homeomorphic to $G/K$, where $K=G_x$ is the stabiliser of $x$. Therefore it suffices to show that $K$ is compact.
The group $K$ preserves the compact ball $\ol B_r(x)$ for every $r>0$.
Now the isometry group of a compact metric space is compact. The group $K$ is the projective limit of all $\Isom(B_r(x))$ and therefore compact.
\end{proof}

\begin{definition}
Suppose that $G\bs X$ is compact and $X$ is equipped with a $G$-invariant Radon measure.
For a continuous, $G$-invariant  function $\phi$ on $X$ we define
$$
\CJ(\phi)=\int_{G\bs X}\phi(x)\, dx.
$$
The integral exists, as $G\bs X$ is compact.
\end{definition}

\section{Lattices}

\begin{definition}
Let $(X,d)$ be a proper metric space. For a point $x\in X$ and a radius $r>0$ we write $B_r(x)$ for the open ball of radius $r$ around $x$.
Let $\Ga$ be a group of isometries on $X$.
The \e{$\Ga$-radius}  of a point $x$ is defined to be
$$
\rho(\Ga,x)=\inf_{1\ne\ga\in\Ga}d(x,\ga x).
$$
The \e{radius} of $\Ga$ is
$$
\rho(\Ga)=\inf_{x\in X}\rho(\Ga,x).
$$
We also define the \e{injectivity radius} at $x$,
$$
\Rinj(\Ga,x)=\sup\big\{r>0: B_r(x)\subset X \text{ injects to }\Ga\bs X\big\}.
$$
We say that $\Ga$ is \e{cocompact}, if the quotient $\Ga\bs X$ is compact.
By a \e{cocompact lattice} we mean a cocompact group $\Ga$ of strictly positive injectivity radius.  
\end{definition}

\begin{lemma}
Suppose that $\rho(\Ga,x)>0$. Then one has
$$
\frac13\rho(\Ga,x)\le \Rinj(\Ga,x)\le \rho(\Ga,x).
$$ 
\end{lemma}

\begin{proof}
If $r>\rho(\Ga,x)$, then $r>\Rinj(\Ga,x)$, so the second inequality follows.
For the first let $r>\Rinj(\Ga,x)$, then we have to show that there is $\ga\in\Ga$, such that $d(x,\ga x)<3r$.
As $r>\Rinj(\Ga,x)$, there are $y\in B_r(x)$ and $\ga\in\Ga^*$ such that $d(y,\ga y)<r$.
Then
\begin{align*}
d(x,\ga x) &\le d(x,y)+d(y,\ga y)+d(\ga y,\ga x)\ <\ r+r+r=3r.\mqed
\end{align*}
\end{proof}

\begin{proposition}
Suppose the group $\Ga\subset G$ has radius $>0$.
Then
\begin{enumerate}[\rm (a)]
\item The group $\Ga$ acts freely on $X$.
\item For every compact set $K\subset X$ the set
$
\big\{\ga\in\Ga: K\cap\ga K\ne\emptyset\big\}
$
is finite.
\item Every point $x\in X$ has a neighbourhood $U$ such that $\ga\in\Ga\sm\{1\}\ \Rightarrow\ \ga U\cap U=\emptyset$.
This means that $\Ga$ is acting \e{properly discontinuously}.
\item Each orbit $\Ga x$ is discrete and closed in $X$, the quotient map $X\to\Ga\bs  X$ is a covering. 
\item Any lattice $\Ga$ is a cocompact discrete subgroup of $G$.
If $X$ admits a lattice, then $G$ is unimodular and $G\bs X$ is compact.
\item If $G\bs X$ is compact and $\Ga\subset G$ a torsion-free, discrete cocompact subgroup of $G$, then $\Ga$ is a lattice on $X$.
\end{enumerate}
\end{proposition}

\begin{proof}
(a) is clear. For (b) assume that the set is not finite and let $x\in K$. By compactness, there is a sequence $\ga_j\in\Ga$ such that $\ga_j x$ converges.
But $d(\ga_ix,\ga_jx)=d(\ga_j^{-1}\ga_ix,x)\ge\rho(\Ga)>0$, a contradiction.

For (c): Take $U=B_{r/2}(x)$, where $r$ is the $\Ga$-radius at $x$.
Let $\ga\in\Ga\sm\{1\}$.
Assume there exists $u\in U\cap \ga U$. Then $d(u,x),d(u,\ga x)<r/2$.
Hence $d(x,\ga x)<r$, which contradicts the notion of the  radius.

For (d): The discreteness of $\Ga x$ follows from the fact, that for $\ga\ne\tau$ in $\Ga$ the distance $d(\ga x,\tau x)=d(x,\ga^{-1}\tau x)$ is $\ge\rho(\Ga)>0$.
The covering part follows from (c).

(e) The orbit $Gx$ is closed by Proposition \ref{prop111}. There is a compact subset $C\subset X$, such that $C\cap Gx$ maps surjectively onto $\Ga\bs (Gx)$.
As $Gx$ is closed, $C\cap Gx$ is compact and so $\Ga\bs (Gx)\cong \Ga\bs G/G_x$ is compact.
By Proposition \ref{prop111}, part (g), the stabiliser group $G_x$ is compact. We infer that $\Ga\bs G$ is compact.
As the map $\Ga\bs X\to G\bs X$ is continuous, the space $G\bs X$ is compact, too.

(f) First we  show that $\Ga\bs X$ is compact. Let $(U_i)_{i\in I}$ be an open covering of $\Ga\bs X$.
The continuous map $\phi:\Ga\bs X\to G\bs X$ has compact fibres $\cong\Ga\bs G/G_x$, $x\in X$.
For each $Gx\in G\bs X$ there is a finite set $E_{Gx}\subset I$ such that 
$U_x=\bigcup_{i\in E_x}U_i\supset\phi^{-1}(Gx)$.
Let $V_x$ denote the set of all $y\in G\bs X$ such that $\phi^{-1}(y)\subset U_x$. Then $V_x=\phi(U_x^c)^c$ is an open neighbourhood of $x$, so there is a finite set $F\subset G\bs X$ such that $G\bs X=\bigcup_{x\in F}V_x$, and so $\Ga\bs X=\bigcup_{x\in F}\phi^{-1}(V_x)\subset\bigcup_{x\in F}\bigcup_{i\in E_x}U_i$ and we have found a finite subcover.
Next, the group $\Ga$ acts freely on $X$, since for $x\in X$ the stabiliser $\Ga_x$ equals $\Ga\cap G_x$, so it is discrete and compact, hence finite. But as $\Ga$ is torsion-free, the stabiliser must be trivial.
Finally, for the positivity of the radius, let $(\ga_n)$ be a sequence in $\Ga$ and $x_n\in X$ such that $d(x_n,\ga_n x_n)$ tends to zero.
By compactness of $\Ga\bs X$ we can assume $x_n\to x$ for some $x\in X$.
Since 
$$
d(x,\ga_n x)\le d(x,x_n)+d(x_n,\ga_nx_n)+\underbrace{d(\ga_nx_n,\ga_nx)}_{d(x_n,x)}
$$
we infer that $d(x,\ga_nx)$ also tends to zero.
As $g\mapsto gx$ is proper by Proposition \ref{prop111} and $\Ga$ is discrete, the sequence $(\ga_n)$ is eventually stationary, which means that $\ga_n=1$ for $n$ large.
We infer that $\Ga$ is a lattice in 
the sense of our definition.
\end{proof}

\vspace{10pt} 

\begin{center}
{\it For the rest of the paper we shall assume that a lattice exists.
}
\end{center}

\vspace{10pt}

\begin{lemma}\label{lem2.3}
Let $(X,d)$ be a proper metric space.
Then for every compact set $K\subset X$ and all $C,r>0$, there exists $T>0$ such that for every lattice $\Ga$ of  radius $\ge r$ one has
$$
\# \big\{ \ga\in\Ga: d(K,\ga K)\le C\big\}\ \le\ T.
$$
\end{lemma}

\begin{proof}
Cover $K$ with  open balls of radius $r/4$, so $K\subset B(p_1)\cup\dots\cup B(p_m)$. 
Then for $\ga\in \Ga$ one has $K\cap\ga K=\bigcup_{1\le i,j\le s}B(p_i)\cap \ga B(p_j)$.
If this is $\ne\emptyset$, there are $i,j$ with $d(p_i,\ga p_j)<r/2$.
If there are more than $m^2$ of such $\ga$, then there exists $\ga\ne\tau$ in $\Ga\sm\{1\}$ and $1\le i,j\le s$ such that $d(p_i,\ga p_j)<r/2$ and the same for $\tau$.
This means that $\frac r2>d(p_i,\tau p_j)=d(\ga\tau^{-1}p_i,\ga p_j)$.
Therefore,
\begin{align*}
r&\le d(p_i,\ga\tau^{-1}p_i)\\
&\le d(p_i,\ga p_j)+d(\ga p_j,\ga\tau^{-1} p_i)<r,
\end{align*}
contradicting the assumptions.
\end{proof}

In the following, we shall use a fixed invariant measure $\mu$  and we shall simply write $dx$ for $d\mu(x)$.
In the same fashion, we treat other measures, which will not lead to confusion as long as it is clear which space we take the  integral over.

\begin{lemma}
Let $\mu$ be an invariant measure on $X$. There exists a uniquely determined Radon measure on $G\bs X$ and for every lattice $\Ga$ there exist uniquely determined  Radon measures on $\Ga\bs X$ and $\Ga\bs G$, such that for every $f\in C_c(X)$ one has
\begin{align*}
\int_Xf(x)\ dx&=\int_{\Ga\bs X}\sum_{\ga\in\Ga}f(\ga x)\ dx\\
&=\int_{G\bs X}\int_{G}f(g x)\ dg\ dx\\
&=\int_{G\bs X}\int_{\Ga\bs G}\sum_{\ga\in\Ga}f(\ga g x)\ dg\ dx\\
\end{align*}
The maps that arise by integration,
\begin{align*}
C_c(X)&\to C(\Ga\bs X),&C_c(X)&\to C(G\bs X),&C(\Ga\bs X)&\to C(G\bs X)
\end{align*}
are surjective.
\end{lemma}

\begin{proof}
We start with the surjectivity assertions.
The first two are covered by Proposition \ref{prop111}.
For the third let $f\in C_c(X)$. By Section 1.5 of \cite{HA2} we get for $x\in X$ that
\begin{align*}
\int_{\Ga\bs G}\sum_{\ga\in\Ga}f(\ga g x)\ dg=\int_Gf(gx)\ dg.
\end{align*}
This means that the diagram
$$
\xymatrix{
C_c(X)\ar[r]\ar[dr]
	&C(\Ga\bs X)\ar[d]\\
	&C(G\bs X)
}
$$
commutes. The surjectivity of the vertical arrow now follows from the surjectivity of the other two.
The rest of the proof runs analogous to Section 1.5 of \cite{HA2}.
\end{proof}

\begin{definition}
A linear operator $A:C_c(X)\to C(X)$ is called an \e{invariant operator}, if for every isometry $\phi$ one has $AL_\phi=L_\phi A$, where $L_\phi(f)=f\aft(\phi^{-1})$.
Let $r>0$.
An invariant operator $A:C_c(X)\to C(X)$ is said to have \e{propagation speed} $\le r$, if 
$$
\supp(Af)\subset U_r(\supp(f))
$$
holds for every $f\in C_c(X)$, where $U_r(M)$ denotes the $r$-neighbourhood of the set $M\subset X$.
\end{definition}

\begin{remark}
\begin{enumerate}[\rm (a)]
\item The map $C_c(X)\to C(\Ga\bs X)$, $f\mapsto f^\Ga$, where 
$$
f^\Ga(x)=\sum_{\ga\in\Ga}f(\ga x)
$$
is surjective. The proof procedes as in Section 1.5 of \cite{HA2}.
\item If $A$ has finite propagation speed and $\Ga$ is a lattice, then there exists a uniquely determined operator $A_\Ga$ on $C(\Ga\bs X)$, such that $A(f)^\Ga=A_\Ga(f^\Ga)$ holds for every $f\in C_c(X)$, where $f^\Ga\in C(\Ga\bs X)$ is given by $f^\Ga(x)=\sum_{\ga\in \Ga} f(\ga x)$.
The map $A\mapsto A_\Ga$ is an algebra homomorphism.
\item If $A$ has a continuous kernel, i.e., there is a continuous map $k_A:X\times X\to\C$, such that $A(f)(x)=\int_Xf(y)\ k_A(x,y)\ dy$, then $A_\Ga$ has the kernel
$$
k_{A_\Ga}(x,y)=\sum_{\ga\in\Ga}k_A(x,\ga y).
$$
Further, the invariance of the operator $A$ is equivalent to
$$
k(gx,gy)=k(x,y)
$$
for all $x,y\in X$, $g\in G$.
The finite propagation speed is equivalent to the kernel being supported in a neighbourhood of the diagonal of the form
$$
U=\big\{ (x,y)\in X\times X: d(x,y)<R\big\}
$$
for some $R>0$.
\end{enumerate}
\end{remark}

\begin{lemma}
Let $\Ga$ be a lattice.
There exists a measurable set $\CF\subset X$ of representatives for $\Ga\bs X$, which has compact closure and for every compact subset $K\subset X$ the set
$$
\big\{\ga\in\Ga: \ga \CF\cap K\ne\emptyset\big\}
$$
is finite, in other words, $K$ is covered by a finite number of translates of $\CF$. Every lattice $\Ga$ is finitely generated.
\end{lemma}

\begin{proof}
Fix a point $p_0\in X$ and let $F_1=F_1(p_0)$ be the set
$$
\big\{ y\in X: d(y,p_0)\le d(\ga y,p_0)\ \forall_{\ga\in\Ga}\big\}=\bigcap_{\ga\in\Ga}H_1(\ga),
$$
where $H_1(\ga)$ is the set of all $y\in X$ with $d(y,p_0)\le d(\ga y,p_0)$.
Since there is $g\in C_c(X)$ with $g^\Ga\equiv 1$, 
it follows that there exists $R>0$ such that every orbit $\Ga y$ meets the ball $B_R(y_0)$. From this ist follows that $F_1$ is compact and that it is the intersection of finitely many $H_1(\ga)$, i.e.,
$$
F_1=\bigcap_{\ga\in S}^nH_1(\ga)
$$
for some finite set $S\subset\Ga$.
Let $H_0(\ga)$ be the set of all $y\in X$ with $d(y,y_0)<d(\ga y,y_0)$ and let $F_0=\bigcap_{j=1}^nH_0(\ga_j)$.
Then $F_0$ is open and if $x\ne y$ are in $F_0$, then they belong to different orbits. On the other hand, $F_1$ meets every orbit and no orbit meets both $F_0$ and $F_1\sm F_0$. 
For $\ga\in S$ let $B(\ga)=F_1\sm H_0(\Ga)$ be the boundary component attached to $\ga$, then $F=F_0\sqcup\[\bigcup_{\ga\in S}B(\ga)\]$.
Fix $\ga_1\in S$ and let $F(\ga_1)$ be the set of all $y\in B(\ga_1)$ which do not lie in any other $B(\ga)$. This is an open subset of a closed subset, hence $F_0\cup F(\ga_1)$ is measurable and it retains the property that it meets every $\Ga$-orbit in at most one point.
Fetch another $\ga_2\in S$ and do the same thing, where this time you also subtract $\Ga F(\ga_1)=\bigcup_{\ga\in\Ga}\ga F(\ga_1)$ the process repeats until in finitely many steps you end up with a set of representatives $\CF=F(p_0)$ with $F_0\subset \CF\subset F_1$ which is a finite union of sets, each of which is  an intersection of a closed and an open set.
So $\CF$ is measurable and has compact closure.
Each compact ball $\ol B_R(p)$, with $R>0$ and $p\in X$ can be covered by finitely many translates $\ga \CF$ and since each compact subset sits in one such ball, the last claim follows. 
\end{proof}

\begin{definition}
Let $X$ be a proper metric space with an invariant measure.
Let $\CA(X)$ denote the algebra of all operators $A$ on $L^2(X)$, such that
\begin{enumerate}[\rm (a)]
\item $A$ is invariant under all isometries,
\item $A$ is given by a continuous integral kernel $k_A$,
\item $A$ has finite propagation speed.
\end{enumerate}
Note that the finite propagation speed and continuity imply that $k$ is bounded.
The adjoint $A^*$ of $A\in\CA(X)$ has kernel
$$
k_{A^*}(x,y)=\ol{k_A(y,x)}.
$$
Hence $\CA(X)$ is a $*$-subalgebra of the algebra of all bounded operators on $L^2(X)$.
\end{definition}

\begin{example}
Let $f:[0,\infty)\to \C$ be continuous of compact support. Then the kernel $k(x,y)=f(d(x,y))$ defines an element of $\CA$.
\end{example}

\begin{definition}
Recall that a \e{finite trace} on a $*$-algebra $\CB$ is a linear map $\tau:\CB\to\C$ such that for all $a,b\in\CB$ one has 
$
\tau(b^*b)\ge 0$ and $\tau(ab)=\tau(ba)
$
\end{definition}

\begin{proposition}
\begin{enumerate}[\rm (a)]
\item For every continuous function $\phi:X\to\C$, which is invariant under isometries, the expression
$$
\CI(\phi)=\frac1{\vol(\Ga\bs X)}\int_{\Ga\bs X}\phi(x)\,dx
$$
does not depend on the lattice $\Ga$.
\item
For every $A\in \CA(X)$ we define 
$$
\tr_\Ga(A)=\int_{\Ga\bs X}k_A(x,x)\ dx.
$$
The expression
$$
\CT(A)=\frac{\tr_\Ga(A)}{\vol(\Ga\bs X)}
$$
 does not depend on the choice of the lattice $\Ga$ and defines a finite trace on $\CA(X)$.
\end{enumerate}
\end{proposition}

\begin{proof}
For (a) let $g\in C_c(X)$ such that $g^\Ga=\phi$.
Then
\begin{align*}
\int_{\Ga\bs X}\phi(x)\, dx
&=\int_X g(x)\,dx\\
&=\int_{G\bs X}g^G(x)\,dx\\
&=\int_{G\bs X}\int_{\Ga\bs G} g^{\Ga}(yx)\,dy\,dx\\
&=\int_{G\bs X}\int_{\Ga\bs G} \phi(yx)\,dy\,dx\\
&=\vol(\Ga\bs G)\int_{G\bs X}\phi(x)\,dx\\
\end{align*}
Setting $\phi=1$ we find $\vol(\Ga\bs X)=\vol(G\bs X)\ \vol(\Ga\bs G)$
and the claim follows.

(b) The independence of $\Ga$ follows from part (a).
It remains to show that $\CT$ is a finite trace.
For this recall that for $A,B\in\CA(X)$ one has 
$$
k_{AB}(x,y)=\int_Xk_A(x,z)k_B(z,y)\ dz
$$
and hence $k_{AB}(x,x)=k_{BA}(x,x)$.
\end{proof}

\begin{definition}
Let $\CA^2$ denote the algebra of all linear combinations of operators of the form $AB$ with $A,B\in\CA$.
The opertors in $\CA$ induce Hilbert-Schmidt operators on $L^2(\Ga\bs X)$ and so the operators in $\CA^2$ are trace class.
\end{definition}

\section{Benjamini-Schramm and spectral convergence}

\begin{definition}\label{defBSforX}
Let $(\Ga_n)$ be a sequence of lattices.
We say that the sequence $(\Ga_n\bs X)$  of metric spaces is \e{Benjamini-Schramm convergent}, or \e{BS convergent} to $X$, if for every $R>0$ the sequence
$$
P_{\Ga_n\bs X}\(\big\{x\in\Ga_n\bs X: \rho(\Ga_n,x)<R\big\}\),
$$
or equivalently the sequence
$$
P_{\Ga_n\bs X}\(\big\{x\in\Ga_n\bs X: \Rinj(\Ga_n,x)<R\big\}\)
$$
tends to zero, where $P_{\Ga_n\bs X}$ is the normalised measure on $\Ga_n\bs X$.
\end{definition}

\begin{definition}
Let $(\Ga_n)$ be a sequence of cocompact lattices in $G$.
We say that the sequence \e{converges spectrally} to $\{1\}$, if for every $A\in\CA^2$ the expression
$$
\frac1{\vol(\Ga_n\bs G)}\tr(A_{\Ga_n})
$$
converges to $\CT(A)$, as $n\to\infty$.
\end{definition}

\begin{definition}\label{def3.3}
A sequence of lattices $(\Ga_n)$ on $X$ is called \e{uniformly discrete}, if the  radii are bounded away from zero.
It is easy to see that this is equivalent to the same notion in \cite{BS-AD}.
\end{definition}

\begin{theorem}\label{thm3.6}
For a uniformly discrete sequence $(\Ga_n)$ of lattices on $X$, the following are equivalent.
\begin{enumerate}[\rm (a)]
\item The sequence $(\Ga_n)$ is spectrally convergent to $\{1\}$.
\item $(\Ga_n\bs X)$ is BS-convergent to $X$.
\end{enumerate}
The implication (a)$\Rightarrow$(b) also holds without the assumption of uniform discreteness.
\end{theorem}

Without the assumption of uniform discreteness, the implication (b)$\Rightarrow$(a) is false in general  as Example \ref{CounterEx} shows.

\begin{proof}
(a)$\Rightarrow$(b): Let $R>0$ and let $k\ge 0$ be the kernel of some $A\in\CA^2$ with $k(x,y)\ge 1$ if $d(x,y)\le R$.
For $g\in G$ and $R>0$ let $F(g)=\big\{ x\in X: d(x,g x)<R\big\}$.
Then
\begin{align*}
\frac1{\mu(\Ga_n\bs X)}\ \mu\(\big\{x\in\Ga_n\bs X: \rho(\Ga_n,x)<R\big\}\)
&=\frac1{\mu(\Ga_n\bs X)}\ \mu\[\Ga_n\bs\bigcup_{1\ne\ga\in\Ga_n}F(\ga)\big\}\]\\
&\le \frac1{\mu(\Ga_n\bs X)}\ \int_{\Ga_n\bs X}\sum_{1\ne\ga\in\Ga_n}k(x,\ga x)\ dx
\end{align*}
and the latter tends to zero, where we haven't used the unifom discreteness.

(b)$\Rightarrow$(a): 
Let $A\in\CA^2$. We have to show that the sequence
$\frac{\tr(A_{\Ga_n})}{\vol(\Ga_n\bs X)}-\CT(A)$ tends to zero.
Let $k$ be the kernel of $A$ and $k_\Ga$ be the one of $A_\Ga$, where $\Ga$ is a lattice.
Let $g\in C_c(X)$ with $g^\Ga\equiv 1$, where $g^\Ga(x)=\sum_{\ga\in\Ga}g(\ga x)$.
Writing $k_1(x,y)=\sum_{\ga\ne 1}k(x,\ga y)$, we have $k_\Ga=k+k_1$. By \cite{HA2}, Chapter 9 we have $\tr(A_\Ga)=\int_{\Ga\bs X}k_\Ga(x,x)\, dx$.
Let $R\ge 0$ be the propagation speed of $A$, let $C=\sup_{x,y\in X}|k(x,y)|$ and let $T$ be the bound of Lemma \ref{lem2.3}.
Noting that $k_1(x,x)$ is $\Ga$-invariant, we get
\begin{align*}
\left|\frac{\tr(A_{\Ga_n})}{\vol(\Ga_n\bs X)}-\CT(A)
\right|&\le\frac1{\vol(\Ga_n\bs X)}\int_{\Ga_n\bs X}|k_1(x,x)|\,dx\\
&\le\frac C{\vol(\Ga_n\bs X)}\int_{\Ga_n\bs X}\#\{\ga\ne 1: d(x,\ga x)\le P\}\,dx\\
&\le CT\ P_{\Ga_n\bs X}\(\{x\in\Ga_n\bs X:\rho(\Ga_n,x)\le R\}\),
\end{align*}
which tends to zero as $n\to\infty$.
\end{proof}

\section{BS-convergence for groups}
\begin{definition} (BS-convergence for groups, see \cite{BS-AD})
We say that a sequence $(\Ga_n)$ of lattices in $G$ is \e{Benjamini-Schramm convergent} or \e{BS-convergent} to $\{1\}$, if for every compact set $C\subset G$ the sequence
$$
P_{\Ga_n\bs G}\(\big\{x\in \Ga_n\bs G: x^{-1}\Ga_n^\star x\cap C\ne \emptyset\big\}\)
$$
tends to zero,  where $P_{\Ga_n\bs G}$ is the normalised invariant volume on $\Ga_n\bs G$.
\end{definition}

\begin{definition}
Let $\CA(G\bs X)$ denote the $\sigma$-algebra generated by all analytic subsets of $G\bs X$, see \cite{Cohn}.
By Theorem 8.5.3 of \cite{Cohn}, there exists a $\CA(G\bs X)$, $\CB(X)$ measurable cross section $\sigma:G\bs X\to X$ to the projection $X\to G\bs X$.
\end{definition}

\begin{theorem}\label{thm4.3}
Let $(\Ga_n)$ be a  sequence of lattices.
Then the following are equivalent.
\begin{enumerate}[\rm (a)]
\item $(\Ga_n\bs X)$ is BS-convergent to $X$.
\item $\Ga_n$ is BS-convergent to  $\{1\}$.
\end{enumerate}
\end{theorem}

\begin{proof}
(a)$\Rightarrow$(b):
For every measurable $A\subset \Ga_n\bs X$ we have 
\begin{align*}
P_{\Ga_n\bs X}(A)
&=\int_{G\bs X} P_{\Ga_n\bs G}\(\{g\in \Ga_n\bs G: gx\in A\}\)\ dx\\
&=\int_{G\bs X} P_{\Ga_n\bs G/G_{\sigma(x)}}\(\{g\in \Ga_n\bs G/G_{\sigma(x)}: gx\in A\}\)\ dx.
\end{align*}
As this tends to zero, we get that $P_{\Ga_n\bs G/G_{\sigma(x)}}\(\{g\in \Ga_n\bs G/G_{\sigma(x)}: gx\in A\}\)$ tends to zero almost everywhere in $x$, so in particular, there exists one such $x$.
Proposition 2.4 of \cite{BS-AD} is formulated for Lie groups only, but the proof applies more generally to give that $\Ga_n$ is BS-convergent to  $\{1\}$.

(b)$\Rightarrow$(a):
Again by Proposition 2.4 of \cite{BS-AD} we get that
the intgerand in the last displayed integral tends to zero pointwise. The claim then follows by dominated convergence.
\end{proof}

\begin{definition}
Let $\Ga$ be a lattice. Then $\Ga\bs G$ is compact.
Let $C_c^\infty(G)$ be the convolution algebra of test functions \cites{Bruhat,traceclass}.
It acts by right convolution through trace class operators
$R(f)$, $f\in C_c^\infty(G)$  on $L^2(\Ga\bs G)$.
In \cite{BS-AD} we called a sequence of lattices $(\Ga_n)$ a \e{Plancherel sequence}, if  $\frac1{\vol(\Ga_n\bs G)}\tr(R(f))$ converges to $f(1)$ for every $f\in C_c^\infty(G)$.
\end{definition}

\begin{corollary}
Let $(\Ga_n)$ be a uniformly discrete sequence of lattices.
Then the following are equivalent.
\begin{enumerate}[\rm (a)]
\item $(\Ga_n)$ is spectrally convergent to $\{1\}$.
\item $\Ga_n$ is BS-convergent to  $\{1\}$.
\item $(\Ga_n\bs X)$ is BS-convergent to $X$.
\end{enumerate}
The implications (a)$\Rightarrow$(b)$\Leftrightarrow$(c) also hold without the uniform discretenes assumption.
\end{corollary}

\begin{proof}
This follows from Theorems \ref{thm3.6} and \ref{thm4.3} together with Theorem 2.6 of \cite{BS-AD}.
\end{proof}

\section{Plancherel sequences}
\begin{definition}
Let $G$ be a metrisable locally compact group.
Then the previous sections apply in the special case $X=G$.
If a sequence $(\Ga_n)$ is spectrally convergent to $\{1\}$ for the space $X=G$, we say that this sequence is a \e{Plancherel sequence}. See \cite{BS-AD} for more on Plancherel sequences.
\end{definition}

\begin{theorem}
Every Plancherel sequence is spectrally convergent.
\end{theorem}

\begin{proof}
Let $\sigma:G\bs X\to X$ be a measurable section.
For $[x]=Gx\in G\bs X$ we equip the compact stabilizer group $G_\sigma[x]$ with the normalised Haar measure and the orbit $G{\sigma[x]}\cong G/G_{\sigma[x]}$ with the induced invariant measure.
For $f\in C_c(G)$ the function $f|_{G\sigma[x]}$ defines a measurable function on the orbit and in this way we get a measurable structure that defines the direct Hilbert integral 
$$
\int_{G\bs X}^\bigoplus L^2\(G\sigma[x]\)\,d\mu([x]).
$$
The map that sends $f\in C_c(G)$ to its induces element of the direct integral, is an isometry and thus extends to an isometry on $L^2(X)$. 
It is onto, since for a given $s\in \int_{G\bs X}^\bigoplus L^2\(G\sigma[x]\)\,d\mu([x])$ we define
$f(g\sigma[x])=s\(gG_{\sigma[x]}\)$, which yields a pre-image of $s$.
We get a unitary $G$-homomorphism
$$
L^2(X)\tto\cong\int_{G\bs X}^\bigoplus L^2\(G\sigma[x]\)\,d\mu([x]).
$$
By $G$-equivariance, we get
a unitary isomorphism
$$
L^2(\Ga_n\bs X)\tto\cong\int_{G\bs X}^\bigoplus L^2\(\Ga_n\bs G\sigma[x]\)\,d\mu([x])=\int_{G\bs X}^\bigoplus L^2\(\Ga_n\bs G/G_{\sigma[x]}\)\,d\mu([x]).
$$
This space is acted upon by the algebra
$$
\int_{G\bs X}^\bigoplus \HS\[L^2\(G_{\sigma[x]}\bs G/G_{\sigma[x]}\)\]\,d\mu([x]),
$$
where $\HS$ denotes the algebra of Hilbert-Schmidt operators.
The same goes for the subalgebra of trace class operators.
For $A\in\CA^2$ with kernel $k$ we find that on the orbit $G\sigma[x]$ it induces an operator $A_{\Ga_n,[x]}$ on $L^2(\Ga_n\bs G/G_{\sigma[x]})$ which is given by the restriction of $k$ to the orbit.
We get that the trace of $A_{\Ga_n}$ equals
$$
\int_{G\bs X}\tr A_{\Ga_n,[x]}\ dx.
$$
Further, $\CT(A)=\int_{G\bs X}k(\sigma[x],\sigma[x])\ dx.
$ On $L^2(\Ga_n\bs G/G_{\sigma[x]})$ the operator $A_{\Ga_n,[x]}$ acts by a continuous kernel. So if $(\Ga_n)$ is Plancherel, we get pointwise convergence of the integrands and by the boundedness of $k$ the integrands can be bounded, so that the dominated convergence theorem implies that $(\Ga_n)$ is spectrally convergent to $\{1\}$.
\end{proof}

\begin{remark}
In total, we get the following picture:
$$
\xymatrix{
(\Ga_n)\text{ is Plancherel}\ar@<-1ex>@{=>}[d]\ar@<1ex>@{=>}[r]\ar@<-1ex>@{<::}[r]
	&(\Ga_n)\text{ spectrally convergent to }\{1\}\ar@<-1ex>@{=>}[d]\\
(\Ga_n)\text{ is }BS\ar@<-1ex>@{::>}[u]\ar@{<=>}[r]
	&(\Ga_n\bs X)\text{ is }BS\ar@<-1ex>@{::>}[u]\\
}
$$
where the dotted arrows mean that this implication holds under the condition of uniform discreteness.
\end{remark}

{\small Mathematisches Institut\\
Auf der Morgenstelle 10\\
72076 T\"ubingen\\
Germany\\
\tt deitmar@uni-tuebingen.de}

\today


\begin{bibdiv} \begin{biblist}

\bib{Abert1}{article}{
   author={Ab\'{e}rt, Mikl\'{o}s},
   author={Csikv\'{a}ri, P\'{e}ter},
   author={Frenkel, P\'{e}ter E.},
   author={Kun, G\'{a}bor},
   title={Matchings in Benjamini-Schramm convergent graph sequences},
   journal={Trans. Amer. Math. Soc.},
   volume={368},
   date={2016},
   number={6},
   pages={4197--4218},
   issn={0002-9947},
   review={\MR{3453369}},
   doi={10.1090/tran/6464},
}

\bib{Abert2}{article}{
   author={Ab\'{e}rt, Mikl\'{o}s},
   author={Csikv\'{a}ri, P\'{e}ter},
   author={Hubai, Tam\'{a}s},
   title={Matching measure, Benjamini-Schramm convergence and the
   monomer-dimer free energy},
   journal={J. Stat. Phys.},
   volume={161},
   date={2015},
   number={1},
   pages={16--34},
   issn={0022-4715},
   doi={10.1007/s10955-015-1309-7},
}

\bib{Abert3}{article}{
   author={Ab\'{e}rt, Mikl\'{o}s},
   author={Csikv\'{a}ri, P\'{e}ter},
   author={Frenkel, P\'{e}ter E.},
   author={Kun, G\'{a}bor},
   title={Matchings in Benjamini-Schramm convergent graph sequences},
   journal={Trans. Amer. Math. Soc.},
   volume={368},
   date={2016},
   number={6},
   pages={4197--4218},
   issn={0002-9947},
   doi={10.1090/tran/6464},
}

\bib{7Samurais}{article}{
   author={Abert, Miklos},
   author={Bergeron, Nicolas},
   author={Biringer, Ian},
   author={Gelander, Tsachik},
   author={Nikolov, Nikolay},
   author={Raimbault, Jean},
   author={Samet, Iddo},
   title={On the growth of $L^2$-invariants for sequences of lattices in Lie
   groups},
   journal={Ann. of Math. (2)},
   volume={185},
   date={2017},
   number={3},
   pages={711--790},
   issn={0003-486X},
   doi={10.4007/annals.2017.185.3.1},
}

\bib{Abert4}{article}{
   author={Abert, Miklos},
   author={Bergeron, Nicolas},
   author={Biringer, Ian},
   author={Gelander, Tsachik},
   title={Convergence of normalized Betti numbers in nonpositive curvature},
   journal={Duke Math. J.},
   volume={172},
   date={2023},
   number={4},
   pages={633--700},
   issn={0012-7094},
   doi={10.1215/00127094-2022-0029},
}


\bib{Anan}{article}{
   author={Anantharaman, Nalini},
   author={Ingremeau, Maxime},
   author={Sabri, Mostafa},
   author={Winn, Brian},
   title={Empirical spectral measures of quantum graphs in the
   Benjamini-Schramm limit},
   journal={J. Funct. Anal.},
   volume={280},
   date={2021},
   number={12},
   pages={Paper No. 108988, 52},
   issn={0022-1236},
   doi={10.1016/j.jfa.2021.108988},
}



\bib{BS-original}{article}{
   author={Benjamini, Itai},
   author={Schramm, Oded},
   title={Recurrence of distributional limits of finite planar graphs},
   journal={Electron. J. Probab.},
   volume={6},
   date={2001},
   pages={no. 23, 13},
   issn={1083-6489},
   doi={10.1214/EJP.v6-96},
}

\bib{Bruhat}{article}{
   author={Bruhat, Fran\c{c}ois},
   title={Distributions sur un groupe localement compact et applications \`a
   l'\'{e}tude des repr\'{e}sentations des groupes $\wp $-adiques},
   language={French},
   journal={Bull. Soc. Math. France},
   volume={89},
   date={1961},
   pages={43--75},
   issn={0037-9484},
}

\bib{Brumley}{article}{
   author={Brumley, Farrell},
   author={Matz, Jasmin},
   title={Quantum ergodicity for compact quotients of ${\rm SL}_d(\Bbb
   R)/{\rm SO}(d)$ in the Benjamini-Schramm limit},
   journal={J. Inst. Math. Jussieu},
   volume={22},
   date={2023},
   number={5},
   pages={2075--2115},
   issn={1474-7480},
   doi={10.1017/S147474802100058X},
}

\bib{Carderi}{article}{
   author={Carderi, Alessandro},
   title={Asymptotic invariants of lattices in locally compact groups},
   journal={C. R. Math. Acad. Sci. Paris},
   volume={361},
   date={2023},
   pages={375--415},
   issn={1631-073X},
   doi={10.5802/crmath.417},
}


\bib{Csik}{article}{
   author={Csikv\'{a}ri, P\'{e}ter},
   author={Frenkel, P\'{e}ter E.},
   title={Benjamini-Schramm continuity of root moments of graph polynomials},
   journal={European J. Combin.},
   volume={52},
   date={2016},
   number={part B},
   part={part B},
   pages={302--320},
   issn={0195-6698},
   doi={10.1016/j.ejc.2015.07.009},
}

\bib{Cohn}{book}{
   author={Cohn, Donald L.},
   title={Measure theory},
   series={Birkh\"{a}user Advanced Texts: Basler Lehrb\"{u}cher. [Birkh\"{a}user
   Advanced Texts: Basel Textbooks]},
   edition={2},
   publisher={Birkh\"{a}user/Springer, New York},
   date={2013},
   pages={xxi+457},
   isbn={978-1-4614-6955-1},
   isbn={978-1-4614-6956-8},
   doi={10.1007/978-1-4614-6956-8},
}

\bib{HA2}{book}{
   author={Deitmar, Anton},
   author={Echterhoff, Siegfried},
   title={Principles of harmonic analysis},
   series={Universitext},
   edition={2},
   publisher={Springer},
   date={2014},
   pages={xiv+332},
   isbn={978-3-319-05791-0},
   isbn={978-3-319-05792-7},
   doi={10.1007/978-3-319-05792-7},
}

\bib{traceclass}{article}{
   author={Deitmar, Anton},
   author={van Dijk, Gerrit},
   title={Trace class groups},
   journal={J. Lie Theory},
   volume={26},
   date={2016},
   number={1},
   pages={269--291},
   issn={0949-5932},
}

\bib{BS-AD}{article}{
   author={Deitmar, Anton},
   title={Benjamini-Schramm and spectral convergence},
   journal={Enseign. Math.},
   volume={64},
   date={2018},
   number={3-4},
   pages={371--394},
   issn={0013-8584},
   doi={10.4171/LEM/64-3/4-8},
}

\bib{Farber}{article}{
   author={Farber, Michael},
   title={Geometry of growth: approximation theorems for $L^2$ invariants},
   journal={Math. Ann.},
   volume={311},
   date={1998},
   number={2},
   pages={335--375},
   issn={0025-5831},
   doi={10.1007/s002080050190},
}

\bib{Fracyk}{article}{
   author={Fraczyk, Mikolaj},
   title={Strong limit multiplicity for arithmetic hyperbolic surfaces and
   3-manifolds},
   journal={Invent. Math.},
   volume={224},
   date={2021},
   number={3},
   pages={917--985},
   issn={0020-9910},
}

\bib{Gelander}{article}{
   author={Gelander, Tsachik},
   title={A view on invariant random subgroups and lattices},
   conference={
      title={Proceedings of the International Congress of
      Mathematicians---Rio de Janeiro 2018. Vol. II. Invited lectures},
   },
   book={
      publisher={World Sci. Publ., Hackensack, NJ},
   },
   date={2018},
   pages={1321--1344},
}



\bib{Kionke}{article}{
   author={Kionke, Steffen},
   author={Schr\"{o}dl-Baumann, Michael},
   title={Equivariant Benjamini-Schramm convergence of simplicial complexes
   and $\ell^2$-multiplicities},
   journal={J. Topol. Anal.},
   volume={14},
   date={2022},
   number={1},
   pages={33--53},
   issn={1793-5253},
   doi={10.1142/S1793525321500126},
}

\bib{LMS}{article}{
   author={Le Masson, Etienne},
   author={Sahlsten, Tuomas},
   title={Quantum ergodicity and Benjamini-Schramm convergence of hyperbolic
   surfaces},
   journal={Duke Math. J.},
   volume={166},
   date={2017},
   number={18},
   pages={3425--3460},
   issn={0012-7094},
   doi={10.1215/00127094-2017-0027},
}

\bib{Levit}{article}{
   author={Levit, Arie},
   title={On Benjamini-Schramm limits of congruence subgroups},
   journal={Israel J. Math.},
   volume={239},
   date={2020},
   number={1},
   pages={59--73},
   issn={0021-2172},
   doi={10.1007/s11856-020-2043-7},
}

\bib{Lotz}{article}{
   author={Lotz, Heinrich P.},
   title={Measurable cross sections},
   journal={Arch. Math. (Basel)},
   volume={41},
   date={1983},
   number={3},
   pages={267--269},
   issn={0003-889X},
   doi={10.1007/BF01194838},
}

\bib{Mohammadi}{article}{
   author={Mohammadi, Amir},
   author={Rafi, Kasra},
   title={Benjamini-Schramm convergence of periodic orbits},
   journal={Geom. Dedicata},
   volume={216},
   date={2022},
   number={5},
   pages={Paper No. 57, 11},
   issn={0046-5755},
   doi={10.1007/s10711-022-00703-9},
}

\bib{Monk}{article}{
   author={Monk, Laura},
   title={Benjamini-Schramm convergence and spectra of random hyperbolic
   surfaces of high genus},
   journal={Anal. PDE},
   volume={15},
   date={2022},
   number={3},
   pages={727--752},
   issn={2157-5045},
   doi={10.2140/apde.2022.15.727},
}

\end{biblist} \end{bibdiv}
\end{document}